\documentclass[12pt]{amsart}

\usepackage{times}      
\usepackage{amssymb}    
\usepackage{amsmath}    
\usepackage{amsthm}
\usepackage{tikz-cd}
\usepackage{placeins}
\usepackage{array}

\usepackage{url}
\makeatletter
\def\url@leostyle{%
  \@ifundefined{selectfont}{\def\UrlFont{\sf}}{\def\UrlFont{\scriptsize\ttfamily}}}
\makeatother
\usepackage[T1]{fontenc}
\usepackage{graphicx}

\newtheorem{theorem}{Theorem}[section]
\newtheorem{corollary}[theorem]{Corollary} 
\newtheorem{lemma}[theorem]{Lemma} 
\newtheorem{proposition}[theorem]{Proposition} 
\newtheorem*{mainthm}{Main Theorem}

\theoremstyle{definition}
\newtheorem{definition}[theorem]{Definition} 

\newtheorem{question}[theorem]{Question}

\def\nnn{\mathbb{N}}
\def\rrr{\mathbb{R}}
\def\zzz{\mathbb{Z}}

\def\rpn{\rrr P^n}

\def\m{\mathcal{M}}
\def\meq{\mathcal{M}_{\mathrm{eq}}}

\def\deph{d_{\mathrm{epH}}}
\def\co{\colon\thinspace}
\def\nsub{\mathrel{\unlhd}}
\def\st{\mathrel{}\middle|\mathrel{}}

\DeclareMathOperator{\dist}{dist}
\DeclareMathOperator{\vol}{vol}
\DeclareMathOperator{\rad}{rad}
\DeclareMathOperator{\isom}{Isom}
\DeclareMathOperator{\diam}{diam}

\newcommand{\Orth}[1]{\mathrm{O}(#1)}

\begin{document}

\title[Convergence of isometries]{Convergence of isometries, with semicontinuity of symmetry of Alexandrov spaces}

\author{John Harvey}
\address{Mathematisches Institut, Universit\"{a}t M\"{u}nster, Einsteinstr. 62, 48149 M\"{u}nster, Germany. }
\email{harveyj@uni-muenster.de}

\subjclass[2010]{Primary: 53C23; Secondary: 53C20} 
\thanks{The author is supported by the SFB ``Groups, Geometry and Actions''.}

\date{\today}

\begin{abstract}
The equivariant Gromov--Hausdorff convergence of metric spaces is studied. Where all isometry groups under consideration are compact Lie, it is shown that an upper bound on the dimension of the group guarantees that the convergence is by Lie homomorphisms. Additional lower bounds on curvature and volume strengthen this result to convergence by monomorphisms, so that symmetries can only increase on passing to the limit.
\end{abstract}

\maketitle

\section{Introduction}

The equivariant Gromov--Hausdorff topology allows one to study the convergence of metric spaces while keeping track of their symmetries.
The definition of this topology involves functions between the metric spaces and between the groups of isometries.
These functions are not required to be continuous, or to be group homomorphisms.

The main theorem of this article is that, assuming that all groups under consideration are compact Lie groups, convergence is always by homomorphisms of Lie groups.

\begin{mainthm}
	Let $(X_i, p_i, G_i)$ be a sequence of pointed group metric spaces, converging to $(X_{\omega}, p_{\omega}, G_{\omega})$ in the equivariant pointed Gromov--Hausdorff topology. 
	Assume that $G_i$ and $G$ are compact Lie groups, with a uniform upper bound on the dimension of the $G_i$.
	Then, for large enough $i$, the functions $G_i \to G_{\omega}$ which demonstrate the convergence may be chosen to be homomorphisms of Lie groups.	
\end{mainthm}

Note that this result does not require any assumptions about the metric spaces themselves (other than that metric balls centered at the distinguished points be relatively compact, which is required to define the topology).
The only assumption is on the groups themselves.

The functions $G_{\omega} \to G_i$ are less tractable.
A simple example such as the convergence of the cyclic group acting on the circle to the action of the full circle, $(S^1,\zzz_p) \to (S^1,S^1)$ as $p \to \infty$, shows that homomorphisms need not exist in the reverse direction.

Where the homomorphisms have a non-trivial kernel, some symmetries are lost.
It is easy to find examples where injectivity fails, either by using spaces with infinite topology (such as Hawaiian earrings) or by shrinking the orbits of the group action so the spaces converge to the orbit space with a trivial group action.

If the spaces $X_i$ are Riemannian manifolds with a lower sectional curvature and volume bound, or, more generally, Alexandrov spaces, then it is shown that the homomorphisms are always injective (Proposition \ref{p:preservation}).

It would be of interest to see whether similar results continue to hold with weaker hypotheses, for example, using non-compact Lie groups, or only considering a Ricci curvature lower bound.

Finally, in Proposition \ref{p:maxvol}, the injectivity result is applied to achieve an understanding of how much symmetry a closed Riemannian manifold can retain when its volume is large relative to its sectional curvature and radius.

\section{Equivariant pointed Gromov--Hausdorff topology}

The equivariant Gromov--Hausdorff topology is a coarse method of defining the convergence of a sequence of metric spaces with isometric group actions. 
It has been used in the study of fundamental groups of certain classes of manifolds \cite{fyannals, enniswei} and in the study of sequences of Riemannian orbifolds \cite{fukaya, horbi}.
A slightly different definition, developed independently by Bestvina \cite{bestvina} and Paulin \cite{paulin}, is used in geometric group theory.
In the setting of geometric group theory, the group is generally fixed. 
However, the definition used in the current work is designed for situations where the object of study is not the group itself, but rather the symmetries of the space.

It is a generalization of the standard Gromov--Hausdorff topology (in fact, a metric) on the set of isometry classes of compact metric spaces \cite{gromov}, which itself generalizes the Hausdorff metric on the closed subsets of a compact metric space.

Let $\m$ be the set of all isometry classes of pointed metric spaces $(X,p)$ (that is, $X$ is a metric space and $p$ is a distinguished point of $X$) such that for each $r > 0$ the open ball $B(p,r)$ is relatively compact. 

Now consider the set of triples $(X,p, \Gamma)$ where $(X,p) \in \m$ and $\Gamma$ is a closed group of isometries acting effectively on $X$. 
Such triples will be referred to as \emph{pointed group metric spaces}.
Say that two pointed group metric spaces are equivalent if they are equivariantly pointed isometric up to an automorphism of the group.
Let $\meq$ be the set of equivalence classes of pointed group metric spaces.

The equivariant pointed Gromov--Hausdorff distance was first defined by Fukaya \cite{fukaya}, and achieved its final form some years later in his work with Yamaguchi \cite{fyannals}. 

If $(X,p,\Gamma) \in \meq$, then let $\Gamma(r) \subset \Gamma$ be $\left\lbrace \gamma \in \Gamma \st \gamma p \in B(p,r) \right\rbrace $.

\begin{definition}Let $(X,p,\Gamma),(Y,q, \Lambda) \in \meq$. An \emph{equivariant pointed Gromov--Hausdorff $\epsilon$--approximation} is a triple $(f,\phi,\psi)$ of functions 
	\begin{align*}
		f \co B(p,1/\epsilon) &\to Y \\
		\phi \co \Gamma(1/\epsilon) &\to \Lambda(1/\epsilon)\\
		\psi \co \Lambda(1/\epsilon) &\to \Gamma(1/\epsilon)
		\end{align*}
	 such that
	\begin{enumerate}
		\item $f(p) =q$;
		\item the $\epsilon$--neighborhood of $B(p,1/\epsilon)$ contains $B(q,1/\epsilon)$;
		\item if $x,y \in B(p,1/\epsilon)$ then $\left| \dist(x,y) - \dist(f(x),f(y)) \right|  < \epsilon$;
		\item if $\gamma \in \Gamma(1/\epsilon)$, and both $x, \gamma x \in B(p,1/\epsilon)$, then $$\dist(f(\gamma x) , \phi(\gamma) f(x) )  < \epsilon;$$
		\item if $\lambda \in \Lambda(1/\epsilon)$, and both $x, \psi(\lambda) x \in B(p,1/\epsilon)$, then $$\dist(f(\psi(\lambda) x) , \lambda f(x) )  < \epsilon.$$
	\end{enumerate}
\end{definition}

Note that these functions need not be morphisms from the relevant category. 
In particular, they need not be continuous, or respect the group structure in any way.

The equivariant pointed Gromov--Hausdorff distance is defined from these approximations by setting  $\deph \left( (X,p,\Gamma),(Y,q, \Lambda) \right)$ equal to the infinum of all $\epsilon$ such that equivariant pointed Gromov--Hausdorff $\epsilon$--approximations exist from $(X,p,\Gamma)$ to $(Y,q, \Lambda)$ and from $(Y,q, \Lambda)$ to $(X,p,\Gamma)$.

By \cite[Proposition 3.6]{fyannals}, given a sequence in $\meq$, if the sequence of underlying pointed metric spaces is convergent in the pointed Gromov--Hausdorff topology, then there is a subsequence which is convergent in the equivariant pointed Gromov--Hausdorff topology.

By \cite[Theorem 2.1]{fukaya}, the sequence of orbit spaces corresponding to a convergent sequence in $\meq$ must itself converge in the usual Gromov--Hausdorff topology.

\section{Approximating symmetries with group homomorphisms}

This section contains the proof of the following theorem.

\begin{theorem}\label{t:equicontinuity}
	Let $(X_i, p_i, G_i)$ be a sequence of pointed group metric spaces in $\meq$, converging to $(X_{\omega}, p_{\omega}, G_{\omega})$ in the equivariant pointed Gromov--Hausdorff topology. 
	Assume that $G_i$ and $G_{\omega}$ are compact Lie groups, with a uniform upper bound on the dimension of the $G_i$.
	Then, for large enough $i$, the functions $G_i \to G_{\omega}$ which demonstrate the convergence may be chosen to be homomorphisms of Lie groups.	
\end{theorem}

The proof of this result relies on the center of mass construction from Grove--Petersen \cite{gphomotopy}, which allows for the construction of continuous maps from discrete ones.
Let us begin this section by reviewing that construction.
The Riemannian manifolds in question will be the compact Lie groups $G_i$ and $G_{\omega}$ with bi-invariant metrics.

Let $(M,g)$ be a complete Riemannian manifold, with $\dim M = n$, $\sec g \geq k$, $\vol (M,g) \geq v$ and $\diam (M,g) \leq D$.

A minimal $\mu$--net for $M$ is defined to be a set of points in $M$ such that the $\mu$--balls cover all of $M$ but the $\frac{\mu}{2}$--balls are disjoint.

It is shown in \cite{gphomotopy} that certain constants $r, R > 0$ and $N \in \nnn$ exist which depend only on $n, k, v$ and $D$, but not on the manifold $M$ itself, so that the following hold:	
\begin{enumerate}
	\item For any minimal $\mu$--net, a ball of radius $\mu$ will have non-empty intersection with at most $N$ of the $\mu$--balls centered on the members of the $\mu$--net.
	$N$ depends only on $n, k$ and $D$.
		
	\item Let $p_1, \ldots p_m \in M$,  and let $\lambda_1, \ldots, \lambda_m > 0$ be weights, so that $\Sigma \lambda_i = 1$. Let $\eta < r(1 + R + \cdots +R^{m-1})^{-1}$. If $\dist(p_i, p_j) < \eta$, $i, j = 1, \ldots , m$, then a center of mass $\mathcal{C}(p_1, \ldots p_m, \lambda_1, \ldots, \lambda_m)$ is defined which depends continuously on the $p_i$ and the $\lambda_i$, is unchanged on dropping any point with weight 0, and satisfies $\dist(\mathcal{C},p_i) < \eta (1 + R + \cdots +R^m)$ for each $i$.
\end{enumerate}

The bi-invariant metric on a Lie group, however, has no place in the definition of equivariant Gromov--Hausdorff convergence.
The natural geometry on the group of isometries derives from how it acts on the metric space.

\begin{definition}
	Let a compact Lie group $G$ act on a pointed metric space $(X,p) \in \m$. 
	Then for each $R > 0$ the \emph{action pseudoseminorm} on $G$ is the continuous map $\left\| \cdot \right\|_R \co G(R) \to \left[ 0, \infty \right) $ given by $$\left\| g \right\|_R = \sup \left\lbrace \dist (x,gx) \st x, gx \in B(p,R)  \right\rbrace.$$
\end{definition}

This is well defined, since for elements of $\m$ the balls centered at $p$ are relatively compact.
$\left\| g \right\|_R = 0$ for any isometry $g$ which fixes the ball of radius $R$.
However, the compactness of $G$ implies that, once $R$ is large enough, the pseudoseminorm vanishes only at the identity.
It is also clear that $\left\| g \right\|_R = \left\| g^{-1} \right\|_R$

The inequality $\left\| gh \right\|_R \leq \left\| g \right\|_R \left\| h \right\|_R$ need not be satisfied, since this supremum might be achieved for some $x \in B(p,R)$ such that $hx \notin B(p,R)$.
However, if $X$ is compact, then once $X \subset B(p,R)$ this inequality is satisfied.

The action pseudoseminorm can be used to define a right-invariant \emph{action pseudosemimetric} on $G$ by $d_R(g,h) = \left\| gh^{-1} \right\|_R$.
(Of course, the construction could also be carried out left-invariantly.)
The preceding comments on the pseudoseminorm easily imply that $d_R$ is non-negative and symmetric.
Once $R$ is sufficiently large, $d_R(g,h)=0 \implies g=h$ and $d_R$ becomes a semimetric.
If $X$ is compact, $d_R$ will also satisfy the triangle inequality for large $R$, and become a true metric, as in the author's earlier work \cite{hequi}.

\begin{proof}[Proof of Theorem \ref{t:equicontinuity}]
	Assume, by passing to a subsequence if necessary, that $\deph\left(  (X_i, p_i, G_i),(X_{\omega}, p_{\omega}, G_{\omega}) \right)  < 1/i$.
	Note that for large enough $i$, $G_i(i)=G_i$.
	
	Choose bi-invariant Riemannian metrics $\sigma_i$ on each $G_i$, for $i = 1, 2, \ldots, \omega$ so that there is a uniform upper bound on the diameter of the groups.
	Let $d_{\sigma_i}$ be the induced distance functions.
	
	Let $d_i$ be the action pseudosemimetric on $G_i$ with respect to $B(p_i,i)$, and let $d_i^{\omega}$ be the action pseudosemimetric on $G_{\omega}$ with respect to $B(p_{\omega},i)$.
	Assume that $i$ is so large that $d_i^{\omega}$ is a semimetric, but note that there can be no similar guarantee for $d_i$.
	
	Consider a sequence of triples $f_i \co B(p_{\omega},i) \to B(p_i,i)$, $\phi_i \co G_{\omega} \to G_i$, $\psi_i \co G_i \to G_{\omega}$, equivariant pointed Gromov--Hausdorff $1/i$--approximations demonstrating the convergence.
	
	\begin{lemma}
		The functions $\psi_i$ may be chosen to be continuous.
	\end{lemma}
	
	\begin{proof}
		For each $i \in \nnn$, let $\nu_i > 0$ be such that $d_{\sigma_i}(g,h) < 2 \nu_i \implies d_i(g,h) < 1/i$.
		Let $A_i$ be a minimal $\nu_i$--net in $(G_i,\sigma_i)$.
		
		Let $\eta_i > 0$ converge to 0, but let each $\eta_i$ be large enough that $d_i^{\omega}(g,h) < 4/i \implies d_{\sigma}(g,h) < \eta_i$.
		This is possible because $d_i^{\omega}$ is increasing with respect to $i$.
		Choose a sequence of minimal $\eta_i$--nets $B_i \subset (G_{\omega},\sigma_{\omega})$.		
		
		By the upper bound on the dimension of $G_i$, and since $k=0$ is a uniform lower bound on the curvature, and $D=1$ is an upper bound on the diameter, there is some $N$ such that for large enough $i$, a ball of radius $\nu_i$ in $G_i$ non-trivially intersects only $N$ of the $\nu_i$--balls centered at elements of $A_i$.
		
		Let $r,R$ be the constants from the center of mass construction \cite{gphomotopy} which are appropriate for $(G_{\omega},\sigma_{\omega})$.
		Write $K=1 + R + \cdots +R^{N}$. 
		Note that for large enough $i$, $3 \eta_i < r/K$.
		
		Define a map $\alpha \co A_i \to B_i$ by mapping $p \in A_i$ to an element of $B_i$ nearest (in the $\sigma_{\omega}$ metric) to $\psi_i(p)$.
		If, for some $p,q \in A_i$, $d_{\sigma_i}(p,q) < 2 \nu_i$, then $d_{\sigma_{\omega}}(\alpha(p),\alpha(q)) < 3 \eta_i$.
		There is an induced map between the Euclidean spaces $\rrr^{A_i} \to \rrr^{B_i}$, where the coordinate associated to any $q \in B_i$ is obtained by summing the co-ordinates for each element of $\alpha^{-1}(q)$.
		
		Then a continuous map $\tilde{\psi_i} \co G_i \to G_{\omega}$ may be defined by composing maps $G_i \to \rrr^{A_i} \to \rrr^{B_i} \to G_{\omega}$.
		
		Let $A_i = \left\lbrace p_i^1 , \ldots, p_i^{\ell} \right\rbrace $ and choose smooth functions $f_i^j \co G_i \to \left[ 0, \infty \right) $, each having their support in the ball of radius $\nu_i$ around $p_i^j$ which sum to 1.
		The map $(G,\rho_i) \to \rrr^{A_i}$ is given by $g \mapsto \left( f_i^j(g)\right) _{j=1}^{l}$.
		Note that points in the image of this map have at most $N$ non-zero coordinates.
		It is possible to assume that $f_i^j (p_i^k) = 0$ whenever $j \neq k$.
		Let us make this assumption, so that elements of $A_i$ are mapped to points with only one non-zero co-ordinate.
		
		The map from $\rrr^{A_i} \to \rrr^{B_i}$ is that induced by $\alpha$, and the map from $\rrr^{B_i} \to (G_{\omega},\rho_{\omega})$ is given by the center of mass construction.
		Note that in the domain points have at most $N$ non-zero coordinates, and the corresponding elements of $B_i$ are at pairwise distance at most $3 \eta_i$, so this map is defined once $i$ is so large that $3 \eta_i < r/K$.
		The set $A_i$ is mapped to $B_i$ by $\tilde{\psi_i}$.
		
		To complete the proof, it is necessary to verify that $\tilde{\psi_i}$ will serve as part of the equivariant pointed Gromov--Hausdorff approximation.
		
		Let $p_i^1, \ldots, p_i^m$ be those elements of $A_i$ within $\nu_i$ of $g$ in the $\sigma_i$ metric. 
		Their images $\tilde{\psi_i}(p_i^1) , \ldots , \tilde{\psi_i}(p_i^m) \in B_i$ are then at most $4\eta_i$ from $\psi_i(g)$ in the $\sigma_{\omega}$ metric.
		The point $\tilde{\psi_i}(g)$ is obtained from the elements of $B_i$ via the center of mass construction, and so is at most $3 \eta_i K$ from those points with non-zero coordinates. 
		This gives a global bound of $\eta_i(3K+4)$ for the difference between $\psi_i$ and $\tilde{\psi_i}$ in the $\sigma_{\omega}$ metric.
		
		Now for each $j \in \nnn$, consider $d_j^{\omega}(\psi_i(g),\tilde{\psi_i}(g))$.
		By continuity of the semimetric $d_j^{\omega}$, for large enough $i$, 
		$$d_{\sigma_{\omega}}(\psi_i(g),\tilde{\psi_i}(g)) < \eta_i(3K+4) \implies d_j^{\omega}(\psi_i(g),\tilde{\psi_i}(g)) < 1/j.$$
		
		Then the triple $f_i \co B(p_{\omega},j/2) \to B(p_i,j/2)$, $\phi_i \co G_{\omega} \to G_i$, $\tilde{\psi_i} \co G_i \to G_{\omega}$ is an equivariant pointed Gromov--Hausdorff $(2/j)$--approximation.
	\end{proof}
	
	Returning to the proof of the theorem, by the monotonicity of the action semimetrics on $G_{\omega}$, and their continuity with respect to the bi-invariant metric, it is clear that for large $i$ the (now assumed to be continuous) map $\psi_i$ will be an almost homomorphism in the sense of Grove--Karcher--Ruh \cite{gkr}. 
	That is to say, for each $g,h \in G_{i}$, $d_{\sigma_{\omega}}(\psi_i(gh)\psi_i(h)^{-1},\psi_i(g)) \leq q$ for a fixed small $q$. 
	By \cite[Theorem 4.3]{gkr}, there is then a continuous group homomorphism within $1.36q$ of $\psi_i$, and again by continuity of $d_j^{\omega}$, for large enough $i$ this homomorphism may be used in place of $\psi_i$.
\end{proof}

\section{Immersed subgroups}

In this section, we address the question of when the homomorphism $\psi_i$ can be chosen to be a monomorphism, so that $G_i$ can be identified with a subgroup of $G_{\omega}$.
In other words, we wish to understand for which convergent sequences in $\meq$ symmetries are preserved, and for which sequences symmetries are lost.

As an application, Proposition \ref{p:maxvol} bounds the symmetries of manifolds of almost maximal volume for their curvature and radius.

Let $H_i$ be the kernel of $\psi_i$. 
Then $(X_i,H_i) \to (X_{\omega},1)$ in $\meq$. 
It follows that $X_i/H_i \to X_{\omega}$ in the Gromov--Hausdorff topology. 
Combining these two convergences, it is clear that the projection map $X_i \to X_i/H_i$ is a Gromov--Hausdorff $\epsilon_i$--approximation for some $\epsilon_i \to 0$, so that the diameter of the orbits of the kernel converges to $0$.

If $H_i$ is non-trivial, then some symmetries are lost, and there are two very natural ways for such actions to arise.
The first is where the sequence of spaces $(X_i,p_i,H_i)$ is obtained by shrinking the group orbits, such as in a Cheeger deformation.

However, this can also occur where the space has an unusual topology.
Consider, for example, the infinite wedge of 2-spheres with a Hawaiian earring topology.
Let $X_{\omega}$ be this space endowed with a metric which has no non-trivial isometries.
Let $X_i$ be isometric to $X_{\omega}$, except that the $i^\textrm{th}$ sphere has a round metric.
Then the isometry group of $X_i$ is $\Orth{2}$, and the limit of $(X_i, \Orth{2})$ is $(X_{\omega}, 1)$.

The following proposition shows that given a lower curvature bound, unless the sequence collapses symmetries are always preserved.
A suitable sense of lower curvature bound is that from Alexandrov geometry.
An Alexandrov space of curvature $\geq k$ is a generalization of a Riemannian manifold with sectional curvature $\geq k$.
Very roughly, it is a metric space in which triangles are ``fatter'' than triangles with the same side-lengths in constant curvature $k$.
The reader is referred to \cite{bgp} for the definition and basic ideas.

The subspace of $\meq$ under consideration is then $$\Omega^n_k = \left\lbrace (X,p,G) \in \meq \st X \textrm{ has curvature } \geq k, \dim X = n, G \textrm{ is compact}\right\rbrace.$$

\begin{proposition}\label{p:preservation}
	Let $(X_i,p_i,G_i)$ be a sequence converging in $\Omega^n_k$. Then the limit group of isometries contains an isomorphic image of $G_i$, for large $i$. 
\end{proposition}

\begin{proof}
	Let $H_i \nsub G_i$ be the kernel of the homomorphism given by Theorem \ref{t:equicontinuity}, so that $(X_i, p_i, H_i) \to (X, p, 1)$. Then $X_i/H_i \to X$. However, $X_i/H_i$ is also a sequence of Alexandrov spaces with curvature bounded below by $k$, and so by the continuity of Hausdorff measure on $\Omega_k^n$ \cite{bgp}, $H_i$ is trivial for large $i$. 
\end{proof}

Let us restrict our attention further to those Alexandrov spaces with compact isometry groups: the space $\mathrm{Alex}^c(k,n) \subset \m$ given by $$\left\lbrace (X,p) \in \m \st X \textrm{ has curvature } \geq k, \dim X = n, \isom X \textrm{ is compact} \right\rbrace.$$
Recalling that the symmetry degree of a metric space is the dimension of its full isometry group, Proposition \ref{p:preservation} has the following corollary:

\begin{corollary}[Semicontinuity of symmetries]\label{c:semicont}
	The symmetry degree is upper semicontinuous on $\mathrm{Alex}^c(k,n)$.
\end{corollary}

The following question is natural:

\begin{question}
	To what extent can the lower curvature bounds in Proposition \ref{p:preservation} and its Corollary \ref{c:semicont} be relaxed?
\end{question}

For example, if the $X_i$ are manifolds with only a lower Ricci curvature bound, that property is not necessarily inherited by the quotient spaces $X_i / H_i$, and so the proof of the Proposition would fail.

\subsection{Voluminous manifolds}

According to Grove and Petersen \cite{gpmaxvol}, the volume of a manifold can be controlled by a lower bound on the sectional curvature and by the radius. The radius of a metric space $X$ is the invariant $\rad X = \min_p \max_q \dist(p,q)$.

\begin{theorem}[Grove--Petersen \cite{gpmaxvol}]\label{t:maxvol}
	Fix a real number $k$, a positive $r$ ($\leq  \pi/\sqrt{k}$ if $k > 0$), and an integer
	$n \geq 2$. Then there is an $w=w(k, r, n)>0$ such that if $M$ is a Riemannian $n$-manifold with $\sec M \geq k$ and $\rad M \leq r$, then $\vol M \leq w(k,r,n)$. Furthermore, if $\vol M$ is sufficiently close to $w(k,r,n)$ then $M$ is topologically either $S^n$ or $\rpn$. 
\end{theorem}

It is clear that an upper bound exists. 
A simple volume comparison gives that $\vol M \leq v_k^n(r)$, the volume of the ball of radius $r$ in the simply connected $n$--dimensional space form of constant curvature $k$. 

In the case of $k > 0$ and $\frac{\pi}{2\sqrt{k}} < \rad M \leq \frac{\pi}{\sqrt{k}}$, the bound is not sharp. 
Here some volume must always be lost, and the manifolds with volume close to the bound $w$ are always homeomorphic to spheres.
However, when $k \leq \frac{\pi}{2\sqrt{k}}$, both topologies admit metrics with volume close to $w$.

Riemannian manifolds can only achieve this volume bound in the case of round spheres and projective spaces.
However, the key observation in proving the theorem is that there is always an Alexandrov space of curvature $\geq k$, radius $\leq r$ and dimension $n$ which achieves the volume $w(k,r,n)$. 
Then any sequence of Riemannian manifolds with volume converging to $w$ must have one of these spaces as its limit, and so be homeomorphic to it by Perelman's Stability Theorem \cite{perstab}.

There are only three types of Alexandrov spaces of maximal volume.
More thorough descriptions of these types are given in \cite{gpmaxvol}. 
Two of these types achieve the volume $v_k^n(r)$.
These are the ``crosscap'' (the disk of radius $r$ with antipodal points on the boundary identified), which is homeomorphic to $\rpn$, and isometric in case $r=\frac{\pi}{2\sqrt{k}}$, and the ``purse'' (the disk of radius $r$, but with boundary points identified by a reflection through a hyperplane), which is homeomorphic to $S^n$.

In case $k>0$ and $r>\frac{\pi}{2\sqrt{k}}$, the maximal volume $w(k,r,n)$ is attained by the ``lemon''.
The lemon is obtained by removing a wedge (that is, a connected component of the complement of two totally geodesic hypersurfaces) from the round sphere $S^n$ of curvature $k$, and then identifying points on the boundary via a reflection.
The lemon is homeomorphic to the sphere, and when $r=\frac{\pi}{\sqrt{k}}$ it is isometric to a round sphere.

The isometry groups of each of these spaces are easily calculated.
Let the dimension of the space be $n$.
In case the space is smooth ($k>0$ and $r=\frac{\pi}{2\sqrt{k}}$ or $\frac{\pi}{\sqrt{k}}$), it is homogeneous with isometry group $\Orth{n+1}$ (with ineffective kernel of order 2 in the $\rpn$ case).

Otherwise, the crosscap has isometry group $\Orth{n}$.
Galaz-Garcia and Guijarro have shown that, just as with Riemannian manifolds, this is the largest possible dimension for the isometry group of a non-homogeneous Alexandrov space \cite{GGGIsom}.
The purse and the lemon, however, have as isometry group the direct product $\zzz_2 \times \Orth{n-1}$.

Combining the information on the isometry groups of these spaces with Proposition \ref{p:preservation}, one obtains the following result on the symmetries of Riemannian manifolds with almost maximal volume.

\begin{proposition}\label{p:maxvol}
	Fix a real number $k$, a positive $r$ (in case $k>0$, requiring $r < \pi/\sqrt{k}$ and  $\neq  \pi/2\sqrt{k}$), and an integer $n \geq 2$. 
	Any Riemannian $n$-manifold $M$ with $\sec (M,g) \geq k$ and $\rad M \leq r$ and $\vol M$ sufficiently close to $w(k,r,n)$ is either $\rpn$ with $\isom M \subset \Orth{n}$ or $S^n$ with $\isom M \subset \zzz_2 \times \Orth{n-1}$. 
\end{proposition}

The largest possible isometry group can always be achieved, simply by carrying out the smoothing construction given in \cite{gpmaxvol} in an equivariant way.

In the $\rpn$ case, the Proposition states merely that $M$ is not homogeneous, by \cite{GGGIsom}.

It is the symmetry gap between the two topological types that is of more interest.
A round $\rpn$ has a larger volume than a round $S^n$ of the same radius.
The gap might be read as an expression of this fact, quantitatively expressing the idea that volume can be maximized in $\rpn$ more naturally, whereas maximizing it in $S^n$ involves forcing volume into the space in a way which destroys more symmetries.

\bibliographystyle{amsabbrv}
\bibliography{C:/Users/John/mybib}

\def\cprime{$'$}
\providecommand{\bysame}{\leavevmode\hbox to3em{\hrulefill}\thinspace}
\providecommand{\MR}{\relax\ifhmode\unskip\space\fi MR }
\providecommand{\MRhref}[2]{%
  \href{http://www.ams.org/mathscinet-getitem?mr=#1}{#2}
}
\providecommand{\href}[2]{#2}
\begin{thebibliography}{10}

\bibitem{bestvina}
M.~Bestvina, \emph{Degenerations of the hyperbolic space}, Duke Math. J.
  \textbf{56} (1988), no.~1, 143--161. \MR{932860 (89m:57011)}

\bibitem{bgp}
Y.~Burago, M.~Gromov, and G.~Perelman, \emph{A. {D}. {A}leksandrov spaces with
  curvatures bounded below}, Uspekhi Mat. Nauk \textbf{47} (1992), no.~2(284),
  3--51, 222, Translation in \emph{Russian Math. Surveys}, 47(2):1--58, (1992).
  \MR{1185284 (93m:53035)}

\bibitem{enniswei}
J.~Ennis and G.~Wei, \emph{Describing the universal cover of a compact limit},
  Differential Geom. Appl. \textbf{24} (2006), no.~5, 554--562. \MR{2254056
  (2008f:53046)}

\bibitem{fukaya}
K.~Fukaya, \emph{Theory of convergence for {R}iemannian orbifolds}, Japan. J.
  Math. (N.S.) \textbf{12} (1986), no.~1, 121--160. \MR{914311 (89d:53083)}

\bibitem{fyannals}
K.~Fukaya and T.~Yamaguchi, \emph{The fundamental groups of almost
  non-negatively curved manifolds}, Ann. of Math. (2) \textbf{136} (1992),
  no.~2, 253--333. \MR{1185120 (93h:53041)}

\bibitem{GGGIsom}
F.~Galaz-Garcia and L.~Guijarro, \emph{Isometry groups of {A}lexandrov spaces},
  Bull. Lond. Math. Soc. \textbf{45} (2013), no.~3, 567--579. \MR{3065026}

\bibitem{gromov}
M.~Gromov, \emph{Groups of polynomial growth and expanding maps}, Inst. Hautes
  \'Etudes Sci. Publ. Math. \textbf{53} (1981), 53--73. \MR{623534 (83b:53041)}

\bibitem{gkr}
K.~Grove, H.~Karcher, and E.~A. Ruh, \emph{Group actions and curvature},
  Invent. Math. \textbf{23} (1974), 31--48. \MR{0385750 (52 \#6609)}

\bibitem{gphomotopy}
K.~Grove and P.~Petersen, \emph{Bounding homotopy types by geometry}, Ann. of
  Math. (2) \textbf{128} (1988), no.~1, 195--206. \MR{951512 (90a:53044)}

\bibitem{gpmaxvol}
K.~Grove and P.~Petersen, V, \emph{Volume comparison \`a\ la {A}leksandrov},
  Acta Math. \textbf{169} (1992), no.~1-2, 131--151. \MR{1179015 (93j:53057)}

\bibitem{hequi}
J.~Harvey, \emph{{Equivariant stability of {A}lexandrov spaces}},
  arXiv:1401.0531 [math.DG], 2014.

\bibitem{horbi}
\bysame, \emph{{Orbifold finiteness under geometric and spectral constraints}},
  arXiv:1401.0739 [math.DG], 2014.

\bibitem{paulin}
F.~Paulin, \emph{Topologie de {G}romov \'equivariante, structures hyperboliques
  et arbres r\'eels}, Invent. Math. \textbf{94} (1988), no.~1, 53--80.
  \MR{958589 (90d:57015)}

\bibitem{perstab}
G.~Perelman, \emph{{{A}lexandrov's spaces with curvatures bounded from below
  {II}}}, Preprint available at
  \url{http://www.math.psu.edu/petrunin/papers/papers.html}, 1991.

\end{thebibliography}

\end{document}